\documentclass[runningheads]{llncs}
\usepackage{graphicx}

\usepackage{amsmath} 
\usepackage{amssymb}  
\usepackage{xcolor}
\usepackage{caption}
\usepackage{graphicx}
\usepackage{mathtools}
\usepackage{caption}
\usepackage{floatrow}
\usepackage{tikz}
\usepackage{mathrsfs}
\usepackage{ulem}

\usepackage{algorithmicx}
    \usepackage[ruled]{algorithm}
    \usepackage{algpseudocode}

\usepackage{pgfplots}
\usepackage{xcolor}
\usetikzlibrary{matrix,arrows,calc,positioning,shapes,decorations.pathreplacing}
\usepackage{graphicx}

\usepackage{amsmath} 

\usepackage{enumerate}
\usepackage[all,tips]{xy}
\SelectTips{cm}{11}

\newcommand{\R}{\mathbb{R}}

\begin{document}
\title{Nonholonomic systems with
inequality constraints\thanks{The authors acknowledge financial support from Grant PID2019-106715GB-C21 funded by MCIN/AEI/ 10.13039/501100011033.}}
%
%
\author{
Alexandre Anahory Simoes\inst{1} \and
Leonardo Colombo\inst{2}}
\authorrunning{A. Anahory and L. Colombo}
%
\institute{
School of Science and Technology, IE University, Spain.
\email{alexandre.anahory@ie.edu}\\ \and
 Centre for Automation and Robotics (CSIC-UPM), Ctra. M300 Campo Real, Km 0,200, Arganda
del Rey - 28500 Madrid, Spain. \email{leonardo.colombo@csic.es}}
\maketitle              
\begin{abstract}
In this paper we derive the equations of motion for nonholonomic systems subject to inequality constraints, both, in continuous-time and discrete-time. The last is done by discretizing the continuous time-variational principle which defined the equations of motion for a nonholonomic system subject to inequality constraints. An example is show to illustrate the theoretical results.
\end{abstract}
\section{Introduction}

Some mechanical systems have a restriction either on the configurations that the system may assume or at the velocities the system is allowed to go. Systems with such restrictions are generally called constrained systems. Nonholonomic  systems  \cite{Bloch,generalized,dLMdD1996,Neimark} are, roughly speaking, mechanical systems with constraints on their velocity that are not derivable from position constraints. They arise, for instance, in mechanical systems that have rolling contact (e.g., the rolling of wheels without slipping) or certain kinds of sliding contact.

Mechanical systems subject to inequality constraints are confined within a region of space with boundary. Collision with the boundary activates constraint forces forbiding the system to cross the boundary into a non-admissible region of space. Inequality constraints appear for instance in the problem of rigid-body collisions, mechanical grasping models and biomechanical locomotion \cite{simoes,lopez}. 

Structure preserving integration of systems with inequality constraints has been addressed in many papers due to its applicability in engineering problems that require nonsmooth techniques (see \cite{kaufman}). In \cite{Fetecau,kaufman}, the authors use variational techniques to deduce the equations of motion and integrators for unconstrained mechanical systems with inequality constraints. In \cite{Fetecau}, the authors extend the space of solutions to a non-autonomous space depending on time in order to remove the non-smoothness during the collision with the boundary. However, in \cite{kaufman}, the authors use nonsmooth analysis to deal with collisions and obtain better structure preservation: for instance, nearly energy conservation.

In this paper we consider nonholonomic systems subject to inequality constraints. The prototype example we examine is that of a wheel rolling without sliding inside a circular table. We extend the technique in \cite{Fetecau} to derive the equations of motion via an adaptation of Lagrange-D'Alembert principle for nonholonomic systems subject to inequality constraints and then we use a modification of discrete Lagrange-d'Alembert principle \cite{CM2001} to derive variational integrators for these systems. Although dealing smoothly with the impact with the boundary, our integrator suffers from the same problems as the ones identified in \cite{kaufman}, in particular, non-conservation of energy during the impact. However, we consider that this paper introduces a first approach to the geometric integration of nonholonomic systems with inequality constraints and motivates the search for other strategies such as DELI equations (see \cite{kaufman}) for nonholonomic systems.

The remainder of the paper is as follows: section $2$ introduces mechanical systems with inequality constraints. In Section $3$, we review nonholonomic systems and introduce the variational principle that gives the equations of motion for nonholonomic systems with inequality constraints. In Section $4$, we develop the discrete counterpart of the results in the preceding section. Finally, in Section $5$, we examine the example of a disk rolling without slipping in a circular table.

\section{Mechanical systems with inequality constraints}

In this paper, we will analyse the dynamics of nonholonomic systems evolving on the configuration manifold $Q$ which are subjected to inequality constraints, i.e., constraints determined by a submanifold with boundary $C$ of the manifold $Q$. The boundary $\partial C$ is a smooth manifold of $Q$ with codimension $1$. Locally, the boundary $\partial C$ is a smooth manifold of the type
$\partial C = \{ q\in Q \ | \ g(q) = 0 \}$
and the manifold $C$ is
$C = \{ q\in Q \ | \ g(q) \leqslant 0 \}$
for some smooth function $g:Q \rightarrow \mathbb{R}$.


In convex geometry, given a closed convex set $K$ of $\mathbb{R}^{n}$, the \textit{polar cone} of $K$ is the set
$K^{p} = \{z\in \mathbb{R}^{n} \ | \ \langle z, y \rangle \leqslant 0, \forall y \in K \}$ (see \cite{brogliato}).
The \textit{normal cone} to $K$ at a point $x\in K$ is given by
$N_{K}(x) = K^{p} \cap \{x\}^{T}$, 
where $\{x\}^{T}$ is the orthogonal subspace to $x$ with respect to the Euclidean inner product.

Based on this construction, we will only use a minimal definition of normal cone suiting the kind of inequality constraints we will be dealing with. Given a submanifold with boundary $C$ as before, the normal cone to a point $q\in \partial C$ is the set
$N_{C}(q)=\{ \lambda dg(q) | \lambda \geqslant 0\}$.
The two definitions match, if $C$ is a closed convex set of $\R^{n}$ with boundary being a hypersurface of dimension $n-1$.


Given a Lagrangian function $L:TQ\to\mathbb{R}$ describing the dynamics, with local coordinates $(q^i,\dot{q}^i)$, $i=1,\ldots,n=\dim Q$, the equations of motion under the presence of  inequality constraints are given by Euler-Lagrange equations
$\displaystyle{\frac{d}{dt}\frac{\partial L}{\partial \dot{q}^{i}} - \frac{\partial L}{\partial q^{i}}= 0}$
whenever the trajectory is in the interior of the constraint submanifold $C\setminus \partial C$. At impact times $t_{i}\in \mathbb{R}$ of the trajectory with the boundary $q(t_{i})\in \partial C$, there is a discontinuity in the state variables of the system, often called a jump. This jump is determined by the equations:
\begin{equation}\label{jump:equations}
    \begin{split}
        & \frac{\partial L}{\partial \dot{q}}|_{t=t_{i}^{+}} - \frac{\partial L}{\partial \dot{q}}|_{t=t_{i}^{-}} \in -N_{C}, \,\,\, E_{L}|_{t=t_{i}^{+}} = E_{L}|_{t=t_{i}^{-}}.
    \end{split}
\end{equation}

\begin{remark}
    We note that a negative sign in the previous equation appears as a consequence of the non-interpenetrability of the constraint.i.e., the mechanical system may not cross the boundary of the admissible variational constraint. We will see exactly how the negative signs appears in the following section.
\end{remark}

Throuhgout the paper, $L$ will be a regular mechanical Lagrangian, i.e., it has the form kinetic minus potential energy \cite{Bloch} and the Legendre transform $\mathbb{F}L:TQ\rightarrow T^{*}Q$ with $\mathbb{F}L(q,\dot{q})=(q, \frac{\partial L}{\partial \dot{q}})$ is a local diffeomorphism.

\section{Nonholonomic systems with inequality constraints}

Assume that there are velocity constraints imposed on the system. We will restrict to constraints that are linear in the velocities. Consider a distribution $\mathcal{D}$ on the configuration space $Q$ describing these constraints, that is, $\mathcal{D}$ is a collection of linear subspaces of $TQ$ ($\mathcal{D}_q\subset T_{q}Q$ for each $q\in Q$). A curve $q(t)\in Q$ will be said to satisfy the constraints if $\dot{q}(t)\in\mathcal{D}_{q(t)}$ for all $t$. Locally, the constraint distribution can be written as
$\mathcal{D}=\{\dot{q}\in TQ|\mu_{i}^{a}(q)\dot{q}^{i}=0,\quad a=1,\ldots,m\}$.


The Lagrange-d'Alembert equations of motion for the system are those determined by $\delta\int_{a}^{b}L(q,\dot{q})dt=0,$ where we choose variations $\delta q(t)$ of the curve $q(t)$ that satisfy $\delta q(a)=\delta(b)=0$ and $\delta q(t)\in\mathcal{D}_{q(t)}$ for each $t\in[a,b]$. Note that here the curve $q(t)$ itself satisfies the constraints. Variations are taken before imposing the constraints and hence, the constraints are not imposed on the family of curves defining the variations.



The nonholonomic equations of motion are obtained from Lagrange-d'Alembert principle and its local expression is
\begin{equation}\label{LdA:eq}
    \begin{split}
        & \frac{d}{dt}\frac{\partial L}{\partial \dot{q}^{i}} - \frac{\partial L}{\partial q^{i}}= \lambda_{a}\mu^{a}_{i}, \,\,\quad \mu_{i}^{a}(q)\dot{q}^{i}=0
    \end{split}
\end{equation}
where $\lambda_{a}$ is a Lagrange multiplier that might be computed using the constraints.




If $C$ is an inequality constraint on the nonholonomic system, then Lagrange-d'Alembert equations are still valid in the interior of $C$. However, the jump conditions must now be changed to accommodate the constraints our system has on velocities as we will see in the following result.

\begin{theorem}
    Let $q:[0,h]\rightarrow Q$ be a nonholonomic trajectory of the nonholonomic system $(L,\mathcal{D})$ subjected to the inequality constraint $q(t)\in C$. Suppose that this system has an impact against the boundary $\partial C$ at the time $t_{i}\in [0,h]$. Then the trajectory satisfies Lagrange-d'Alembert equations \eqref{LdA:eq} in the intervals $[0, t_{i}^{-}[$ and $]t_{i}^{+},h]$ and at the impact time $t_{i}$, the following conditions hold:
    \begin{equation}\label{nh:jump:equations}
        \begin{split}
            & \frac{\partial L}{\partial \dot{q}}|_{t=t_{i}^{+}} - \frac{\partial L}{\partial \dot{q}}|_{t=t_{i}^{-}} \in -N_{C}\cup \mathcal{D}^{o},\,\,\, E_{L}|_{t=t_{i}^{+}} = E_{L}|_{t=t_{i}^{-}},\,\,\,    \dot{q}(t_{i}^{+}) \in \mathcal{D}_{q(t_{i}^{+})},
        \end{split}
    \end{equation} where $\mathcal{D}^{o}$ denotes the anihilator of the distribution $\mathcal{D}$.
\end{theorem}

\begin{proof}
    The Lagrange-d'Alembert principle for systems with impacts is defined on the path space
    $\Omega = \{ (c,t_{i}) \ | \ c:[0,h]\rightarrow Q \text{ is a smooth curve and } t_{i}\in \mathbb{R}\}$.

    If the mapping $\mathcal{A}:\Omega \rightarrow \mathbb{R}$ is the action then, the Lagrange,d'Alembert principle states that the derivative of the action should annihilate all variations $(\delta q, \delta t_{i})$ with $\delta q \in \mathcal{D}$. Since,
\begin{equation*}
        \delta \mathcal{A} = \int_{0}^{t_{i}^{-}} \left[ \frac{\partial L}{\partial q^{i}} - \frac{d}{dt}\frac{\partial L}{\partial \dot{q}^{i}}\right] \delta q \ dt + \int_{t_{i}^{+}}^{h} \left[ \frac{\partial L}{\partial q^{i}} - \frac{d}{dt}\frac{\partial L}{\partial \dot{q}^{i}}\right] \delta q \ dt - \left[ \frac{\partial L}{\partial \dot{q}^{i}}\delta q + L \delta t_{i} \right]_{t_{i}^{-}}^{t_{i}^{+}}
\end{equation*}
the fact that Lagrange-d'Alembert equations hold on the intervals $[0, t_{i}^{-}[$ and $]t_{i}^{+},h]$ follows from the application of the fundamental theorem of calculus of variations together with the fact that $\delta q \in \mathcal{D}$. The jump condition follows from the fact that $q(t_{i})\in \partial C$ from where
$\delta (q(t_{i})) \in T(\partial C) \implies \delta q (t_{i}) + \dot{q}(t_{i})\delta t_{i} \in T(\partial C)$.

The variations satisfying the previous equation are spanned by variations $\delta q (t_{i}) \in T(\partial C)$ and $\delta t_{i} = 0$ or $\delta t_{i} = 1$ and $\delta q (t_{i}) = - \dot{q}(t_{i})$. From the latter we immediately deduce that
$\displaystyle{\left[ \frac{\partial L}{\partial \dot{q}^{i}}\dot{q} - L \right]_{t_{i}^{-}}^{t_{i}^{+}} = 0}$,
which is the energy conservation condition in the jump equations. From $\delta t_{i}=0$, we get that 
$$\frac{\partial L}{\partial \dot{q}}|_{t=t_{i}^{+}} - \frac{\partial L}{\partial \dot{q}}|_{t=t_{i}^{-}} = \mathbb{F}L|_{t=t_{i}^{+}} - \mathbb{F} L|_{t=t_{i}^{-}}$$
annihilates $\delta q$ if either it is on the annihilator of $\partial C$ or it belongs to the annihilator of the distribution $\mathcal{D}$, since $\delta q$ is in $T(\partial C)\cap \mathcal{D}$.

Now, in order to have $g(q(t))\leqslant 0$ and since $g(q(t_{i}))=0,$ we must have that $dg(\dot{q}(t_{i}^{-}))\geqslant 0$ and $dg(\dot{q}(t_{i}^{+})) \leqslant 0$, otherwise $q(t)$ would violate the inequality constraint. Noting that $(\partial C)^{o}$ is the union of $N_{C}$ and $-N_{C}$, let us show that $\mathbb{F}L|_{t=t_{i}^{+}} - \mathbb{F} L|_{t=t_{i}^{-}}$ is not in $N_{C}$. Suppose it was on the normal cone then
$$\mathbb{F}L|_{t=t_{i}^{+}} - \mathbb{F} L|_{t=t_{i}^{-}} = \lambda dg(q_{i}), \ \lambda \geqslant 0.$$
This is equivalent to $\dot{q}(t_{i}^{+}) - \dot{q}(t_{i}^{-}) = \lambda (\mathbb{F}L)^{-1}(dg(q_{i}))$. 
Applying $dg(q_{i})$ to both sides of the equation we get
$dg(\dot{q}(t_{i}^{+})) = dg(\dot{q}(t_{i}^{-})) + \lambda dg(q_{i})( (\mathbb{F}L)^{-1}(dg(q_{i})) )$, 
where the right-hand side is greater or equal than $0$, which is not possible. Therefore,
$\mathbb{F}L|_{t=t_{i}^{+}} - \mathbb{F} L|_{t=t_{i}^{-}} \in -N_{C}.$ This is precisely the first jump equation. The third one follows from the nonholonomic constraints.\hfill$\square$
\end{proof}



\begin{remark}
    The previous jump equations are in accordance with the equations obtained in \cite{Clark} from Weierstrass-Erdemann conditions for impacts.
\end{remark}

\section{Nonholonomic integrators for systems with inequality constraints}


We review here the formalism proposed in \cite{CM2001} (see also \cite{Cortes}) which gives rise to the discrete Lagrange-d'Alembert equations. Consider the discrete Lagrangian function $L_{d}:Q\times Q \times \mathbb{R} \rightarrow\mathbb{R}$ on the discrete velocity space $Q\times Q$. Let $\mathcal{D}$ be a distribution on $Q$ and consider a discrete constraint space $\mathcal{D}_{d}\subseteq Q\times Q$ whose dimension agrees with that of the distribution $\mathcal{D}$ as a submanifold of $TQ$, $\text{dim} \mathcal{D}_{d}=\text{dim} \mathcal{D}$ and such that the diagonal set of $Q\times Q$ is contained in the discrete constraint space, $(q,q)\in \mathcal{D}_{d}$ for all $q\in Q$.

Then, the discrete Lagrange-d'Alembert principle asserts that the discrete flow is a critical value of the discrete action map $S_d:C_d^{N} (Q)\rightarrow \R$, which is given by  $\displaystyle{S_d(q_d)=\sum_{k=0}^{N-1} L_{d}(q_k,q_{k+1},h)}$, but this time we impose the restriction $\delta q_{k}\in \mathcal{D}_{q_{k}}$, that is, the infinitesimal variation of the sequence must lie in the constraint distribution.  Lagrange-d'Alembert principle states the following:

\begin{definition}[Discrete Lagrange-d'Alembert principle]
	The discrete flow of the discrete nonholonomic Lagrangian system determined by the discrete Lagrangian function $L_{d}$, the distribution $\mathcal{D}$ and the discrete constraint space $\mathcal{D}_{d}$ satisfies the constraint $(q_{k},q_{k+1})\in \mathcal{D}_{d}$ for all $k\in \{0,...,N-1\}$ and is a critical value of the discrete action map $S_{d}$ among all variations of sequences with fixed end-points whose infinitesimal variations satisfy $\delta q_{k}\in \mathcal{D}_{q_{k}}$.
\end{definition}

As it happens with its continuous counterpart, the application of the discrete Lagrange-d'Alembert principle leads to a set of equations which will be the necessary and sufficient conditions to find critical values subordinated to the imposed restrictions. Assume in the following that $\mu^{a}\in \Omega^{1}(Q)$ with $a=1,...,n-k$ are 1-forms on $Q$ defining the distribution $\mathcal{D}=\{v\in TQ \ | \ \mu^{a}(v)=0\}$ and $\mu_{d}^{a}$ are a set of $n-k$ functions on $Q\times Q$ whose zero set is the discrete constraint space $\mathcal{D}_{d}$.

A sequence $\{q_{k}\}_{k=1}^{N}$ of points in $Q$ satisfies the discrete-Lagrange d'Alembert principle for the triple $(L_{d},\mathcal{D},\mathcal{D}_{d})$ if and only if it satisfies the equations
	\begin{equation}\label{DLA}
	\begin{split}
	& D_{2}L_{d}(q_{k-1},q_{k},h)+D_{1}L_{d}(q_{k},q_{k+1},h)=\lambda_{a}\mu^{a}(q_{k}),\,\,\,
	 \mu^{a}_{d}(q_{k},q_{k+1})=0.
	\end{split}
	\end{equation}


\subsection{Discrete equations with inequality constraints}

Consider a sequence of points $\{q_{k}\}_{k=0}^{N}$ contained in the inequality constraint set $C$. This sequence shall be considered as a discretization of a continuous smooth curve $q:[0, Nh]\rightarrow C$ satisfying $q(kh)=q_{k}$. In fact, since we will also need the time sequence we will use the notation $t_{k}=kh$. Now suppose that this curve has an impact against the boundary of $C$ at the point $\bar{q}\in \partial C$ and that this impact occurs at time $\bar{t} := t_{i-1} + \alpha h$, for some $\alpha \in ]0,1[$ and $i\in\{1, \dots, N\}$, so that $t_{i-1}<\bar{t}<t_{i}$. We will also use the notation $q_{d}:\{t_{0}, \dots, t_{i-1}, \bar{t}, t_{i}, \dots, t_{N}\} \rightarrow Q$ to denote the sequence $\{q_{k}\}_{k=0}^{N}\cup \{\bar{q}\}$ in functional notation.

In the following we will consider the discrete path space $\mathcal{M}_{d}$ formed by sequences such as the one described in the last paragraph:
\begin{equation*}
    \mathcal{M}_{d} = ]0,1[ \times \{ q_{d}:\{t_{0}, \dots, t_{i-1}, \bar{t}, t_{i}, \dots, t_{N}\} \rightarrow Q \ | \ \bar{q}\in \partial C \ \}.
\end{equation*}
This discrete path space is actually a manifold since it is isomorphic to  $]0,1[\times Q \times \cdots \times \partial C \times \cdots \times Q$.

To obtain the discrete equations of motion of a nonholonomic system under inequality constraints we must find the number $\alpha \in ]0,1[$ and the sequence $q_{d}$ such that the differential of the action $S_d:\mathcal{M}_{d}\rightarrow \R$, given by
\begin{equation*}
    \begin{split}
        S_d(\alpha, q_d)=\sum_{k=0}^{i-2} L_{d}(q_k,q_{k+1},h) & + \sum_{k=i}^{N-1} L_{d}(q_k,q_{k+1},h) + \\
        & L_{d}(q_{i-1}, \bar{q}, \alpha h) + L_{d}(\bar{q}, q_{i}, (1-\alpha)h)
    \end{split}
\end{equation*}
annihilates variations $(\delta \alpha, \delta q_{d}) \in T \mathcal{M}_{d}$ satisfying $\delta q_{d} \in \mathcal{D}$, i.e., $\delta q_{i}, \delta \bar{q} \in \mathcal{D}$, where $\delta q_{d} = (\delta q_{1}, \dots, \delta q_{i-1}, \delta \bar{q}, \delta q_{i}, \dots, \delta q_{N})$, $\delta q_{0} = \delta q_{N} = 0$ and $(q_{i-1},\bar{q})$, $(\bar{q}, q_{i})$, $(q_{k}, q_{k+1})\in \mathcal{D}_{d}$ for all $k\neq i-1$.

\begin{theorem}
    Let $\{q_{k}\}$ be a nonholonomic discrete trajectory of the nonholonomic system $(L_{d},\mathcal{D}, \mathcal{D}_{d})$ subjected to the inequality constraint $q_{k}\in C$. Suppose that this system has an impact against the boundary $\partial C$ at the time $\bar{t}\in [0,Nh]$. Then the trajectory satisfies discrete Lagrange-d'Alembert equations \eqref{DLA} for $k \neq i-1, i, $ and at the impact time $\bar{t}$, the following conditions hold:
    \begin{equation}\label{dnh:jump:equations}
        \begin{split}
            & D_{2}L_{d}(q_{i-2},q_{i-1},h) + D_{1}L_{d}(q_{i-1}, \bar{q}, \alpha h) = \lambda_{a}\mu^{a}(q_{i-1}) \\
            & D_{2}L_{d}(\bar{q},q_{i},(1-\alpha)h) + D_{1}L_{d}(q_{i}, q_{i+1}, h) = \tilde{\lambda}_{a}\mu^{a}(q_{i}) \\
            & D_{2}L_{d}(q_{i-1},\bar{q},\alpha h) + D_{1}L_{d}(\bar{q},q_{i}, (1-\alpha) h)\in -N_{C}(\bar{q})\cup \mathcal{D}^{o}_{\bar{q}} \\
            & D_{3}L_{d}(q_{i-1},\bar{q},\alpha h) - D_{3}L_{d}(\bar{q},q_{i}, (1-\alpha) h) = 0 \\
            & \bar{q} \in \partial C, \ (q_{i-1},\bar{q}), (\bar{q}, q_{i}), (q_{i},q_{i+1}) \in \mathcal{D}_{d}.
        \end{split}
    \end{equation}
\end{theorem}

\begin{proof}
    First, we compute the variations of the action map and using $\delta q_{0} = \delta q_{N} = 0$, we obtain that
    \begin{equation*}
        \begin{split}
            \delta S_{d} (\alpha, q_{d}) \cdot (\delta \alpha, \delta q_{d}) = & \sum_{k=1}^{i-2} \left[ D_{2}L_{d} (q_{k-1},q_{k},h)+D_{1}L_{d}(q_{k},q_{k+1},h)\right] \delta q_{k} \\
            + & \sum_{k=i+1}^{N-1} \left[ D_{2}L_{d} (q_{k-1},q_{k},h)+D_{1}L_{d}(q_{k},q_{k+1},h)\right] \delta q_{k} \\
            + & \left[ D_{2}L_{d}(q_{i-2},q_{i-1},h) + D_{1}L_{d}(q_{i-1}, \bar{q}, \alpha h) \right] \delta q_{i-1} \\
            + & \left[ D_{2}L_{d}(q_{i-1},\bar{q},\alpha h) + D_{1}L_{d}(\bar{q},q_{i}, (1-\alpha) h) \right] \delta \bar{q} \\
            + & \left[ D_{2}L_{d}(\bar{q},q_{i},(1-\alpha)h) + D_{1}L_{d}(q_{i}, q_{i+1}, h) \right] \delta q_{i} \\
            + & h \left[ D_{3}L_{d}(q_{i-1},\bar{q},\alpha h) - D_{3}L_{d}(\bar{q},q_{i}, (1-\alpha) h) \right] \delta \alpha.
        \end{split}
    \end{equation*}
Using the facts that $\delta q_{d} \in \mathcal{D}$ and $(q_{k}, q_{k+1})\in \mathcal{D}_{d}$ for all $k\neq i-1$, we immediately get Lagrange-d'Alembert equations \eqref{DLA} for $k\neq i-1, i$ and the first two equations in \eqref{dnh:jump:equations}. From $\delta \bar{q} \in \mathcal{D} \cap T(\partial C)$ we conclude that  
$$D_{2}L_{d}(q_{i-1},\bar{q},\alpha h) + D_{1}L_{d}(\bar{q},q_{i}, (1-\alpha) h) \in \mathcal{D}^{o}\cup (-N_{C}),$$
where we used the fact that the jump during the impact must produce a new point $q_{i} \in C$. Finally, since $\delta \alpha$ is arbitrary we get the last equation in \eqref{dnh:jump:equations}.\hfill$\square$
\end{proof}
\begin{remark}
The Lagrange-d'Alembert equations in the inequality constraint setting may be used as a numerical method to integrate the equations of motion. Given two initial points $q_{0}, q_{1}$ satisfying the discrete constraint $\mathcal{D}_{d}$, we may use discrete Lagrange-d'Alembert equations to obtain the sequence $\{q_{0},\dots, q_{i-1}\}$. Then we use the first equation in \eqref{dnh:jump:equations} to obtain $\bar{q}$ from where we may use the third to obtain $q_{i}$ and then the second to obtain $q_{i+1}$. Then, we may use again Lagrange-d'Alembert equations to integrate the remaining points.\end{remark}

\section{Example}

We will consider the motion of a vertical rolling disk without sliding in a circular table and use our previous construction to find an integrator for the impact time.

The vertical rolling disk is  described by four coordinates: $x, y$ determine the position of the center of mass in the table, $\theta$ indicates the angle that a fixed point in the disk border makes with the vertical axis and $\varphi$ indicates the orientation of the disk with respect to the $x$-axis. Below $m$ is the mass of the disk, $I$ and $J$ are its moments of inertia and $R$ is the disk radius. The dynamics of the vertical rolling disk with unit radius is given by the Lagrangian function
\begin{equation*}
    L=\frac{m}{2}\left( \dot{x}^2 + \dot{y}^2\right) + \frac{I}{2}\dot{\theta}^2 + \frac{J}{2}\dot{\varphi}^2
\end{equation*}
together with the non-slipping constraints $\dot{x}=R\dot{\theta}\cos \varphi$, $\dot{y}=R\dot{\theta}\sin \varphi$ generating the distribution $\mathcal{D}= \left\langle \left\{ \frac{\partial}{\partial \theta} + \cos \varphi \frac{\partial}{\partial x} + \sin \varphi \frac{\partial}{\partial y}, \frac{\partial}{\partial \varphi} \right\} \right\rangle.$

The non-interpenetrability condition of the circular table with radius $a$ implies the inequality constraints $C_{+}$ and $C_{-}$ determined by
$$C_{\pm} = \{ (x,y,\theta,\varphi) \ | (x \pm R\cos \varphi)^{2}+( y \pm R\sin \varphi)^{2} \leqslant a^{2} \}$$
which express the fact that the disk, counting with its radius, cannot leave the table. The proposed integrator uses discrete Lagrange-d'Alembert integrator until the first impact. When we first obtain a non-admissible solution, we switch to the impact equations to determine the exact impact point. Afterwards, we return to DLA scheme until the next impact. The discrete Lagrangian used is
$$L_{d}=\frac{m}{2h}\left((x_{1}-x_{0})^{2} + (y_{1}-y_{0})^{2}\right) + \frac{I}{2h}(\theta_{1}-\theta_{0})^{2} + \frac{J}{2h}(\varphi_{1}-\varphi_{0})^{2}$$
and the discrete constraint was
$$x_{1}-x_{0} = R \cos(\frac{\varphi_{1}+\varphi_{0}}{2})(\theta_{1}-\theta_{0}), \quad y_{1}-y_{0} = R \sin(\frac{\varphi_{1}+\varphi_{0}}{2})(\theta_{1}-\theta_{0}).$$
Near an impact point with, for instance, $\partial C_{+}$ the first equation in \eqref{dnh:jump:equations} is used together with the constraints $\bar{q}\in \partial C_{+}$ and $(q_{i-1},\bar{q})\in \mathcal{D}_{d}$ to obtain $\bar{q}, \alpha$ and the multiplier $\lambda_{a}$. Next, using these variables, we use the third  and fourth equations in \eqref{dnh:jump:equations}, together with the constraint $(\bar{q},q_{i})\in \mathcal{D}_{d}$ to obtain $q_{i}$ and the multipliers.
Finally, we use the previous variables to obtain $q_{i+1}$ and the multiplier $\tilde{\lambda}_{a}$ from the second equation in \eqref{dnh:jump:equations} together with the discrete constraints.






Our integrator behaves as expected before the first impact. During the impact time, it is able to deal sucessfully with the impact and produce an admissible trajectory. However, it has proven to be unable to preserve energy thus introducing artificial energy drift into the problem.

This fact shows that further study should be made to look for algorithms that preserve the impact structure. In particular, a promising direction is the extension of DELI integrators to nonholonomic systems.


\end{document}